\documentclass[a4paper, 12pt]{article}

\usepackage[sort&compress]{natbib}
\bibpunct{(}{)}{;}{a}{}{,} 

\usepackage{amsthm, amsmath, amssymb, mathrsfs, multirow, url, subfigure}
\usepackage{graphicx} 
\usepackage{ifthen} 
\usepackage{amsfonts}
\usepackage[usenames]{color}
\usepackage{fullpage}

\theoremstyle{plain}

\newtheorem{prop}{Proposition}

\theoremstyle{definition}

\theoremstyle{remark} 
\newtheorem{remark}{Remark}
\newtheorem*{astep}{A-step}
\newtheorem*{pstep}{P-step}
\newtheorem*{cstep}{C-step}

\newcommand{\prob}{\mathsf{P}}

\newcommand{\lpi}{\underline\Pi} 
\newcommand{\upi}{\overline\Pi} 

\newcommand{\lPi}{\underline\Pi} 
\newcommand{\uPi}{\overline\Pi} 

\newcommand{\bel}{\underline\Pi} 
\newcommand{\pl}{\overline\Pi}

\newcommand{\unif}{{\sf Unif}}
\newcommand{\nm}{{\sf N}}
\newcommand{\expo}{{\sf Exp}}

\newcommand{\chisq}{{\sf Chisq}}

\newcommand{\RR}{\mathbb{R}}

\newcommand{\UU}{\mathbb{U}}

\newcommand{\YY}{\mathbb{Y}} 

\newcommand{\C}{\mathscr{C}}

\renewcommand{\S}{\mathcal{S}}

\newcommand{\sign}{\mathrm{sign}}

\newcommand{\stgeq}{\geq_{\text{st}}}

\title{Inferential models and possibility measures\footnote{This manuscript has been prepared for inclusion in the forthcoming {\em Handbook on Bayesian, Fiducial, and Frequentist (BFF) Inferences}, edited by J.~Berger, X.-L.~Meng, N.~Reid, and M.~Xie.}}
\author{
Chuanhai Liu\footnote{Department of Statistics, Purdue University, {\tt chuanhai@purdue.edu}} \quad and \quad Ryan Martin\footnote{Department of Statistics, North Carolina State University, {\tt rgmarti3@ncsu.edu}}
}
\date{\today}

\begin{document}

\maketitle 

\begin{abstract}
The inferential model (IM) framework produces data-dependent, non-additive degrees of belief about the unknown parameter that are provably valid.  The validity property guarantees, among other things, that inference procedures derived from the IM control frequentist error rates at the nominal level.  A technical complication is that IMs are built on a relatively unfamiliar theory of random sets.  Here we develop an alternative---and practically equivalent---formulation, based on a theory of {\em possibility measures}, which is simpler in many respects.  This new perspective also sheds light on the relationship between IMs and Fisher's fiducial inference, as well as on the construction of optimal IMs.  

\smallskip

\end{abstract}

\section{Introduction}
\label{S:intro}

Broadly speaking, statistics aims to quantify uncertainty about relevant unknowns based on observed data and, therefore, plays a fundamental role in science.  More precisely, the goal of statistical inference, as we see it, is to take inputs---including data, posited statistical model (if any), and prior information (if any)---and return output in the form of meaningful numerical degrees of belief for relevant hypotheses concerning the unknowns.  But despite its importance, there is still no consensus in the statistical community about how the inputs should be turned to output, what mathematical form the output should take, and what statistical properties they ought to satisfy.  Reading between the lines, these questions make up what \citet{efron.cd.discuss} called the ``most important unresolved problem in statistical inference.''  The {\em BFF} or {\em Bayes, Fiducial, and Frequentist} group \citep[e.g.,][]{xl.bff.2017} was created, circa 2014, largely to foster research efforts that could help resolve this unresolved problem.  While the depth and breadth of the group's contributions has been remarkable, the diversity of perspectives in this volume reveals that we are no closer to a resolution than we were in 2014.  In light of the replication crisis in science \citep[e.g.,][]{camerer2018}, the general confusion about and lack of trust in statistics \citep[e.g.][]{mcshane.etal.rss, pvalue.ban}, and new competition from applied mathematics, computer science, and engineering for resources and opportunities, resolving these problems is more important now than ever.   


Given that a theory of statistical inference is intended to be a foundation on which methods will be developed for solving scientific problems, it is essential to consider the totality of applications in order to ensure a base level of reliability or replicability.  As \citet{reid.cox.2014} put it
\begin{quote}
{\em it is unacceptable if a procedure...of representing uncertain knowledge would, if used repeatedly, give systematically misleading conclusions.}
\end{quote}
Central to the development of our {\em inferential model} (IM) framework is the principle that belief assignments must be calibrated, or {\em valid} in the technical sense of Equation \eqref{eq:valid.bel} or \eqref{eq:valid.pl} below.  Our validity condition equips the belief assignments with a necessary external qualification that determines an objective scale for interpretation.  A startling realization was that validity, in the strong sense that we advocate, cannot be achieved by a framework whose belief assignments satisfy the mathematical properties of additive probabilities.  Therefore, if one strictly adheres to validity, then a departure from virtually all of the existing schools of thought is required, and the IM framework---first described in the sequence of papers \citet{imbasics, imcond, immarg} and then in our monograph \citep{imbook}---provides a guide into the world of statistical inference based on non-additive or imprecise probabilities.  

After recognizing the apparently fundamental role that imprecise probability plays in statistical inference, we spent the last year digging into the imprecise probability literature in hopes of better understanding this important new connection.  The present paper, the first of several resulting from these recent efforts, offers an alternative perspective on and construction of IMs.  An essential and distinguishing feature of the IM framework is its use of random sets to quantify uncertainty about the unobserved value of a certain auxiliary variable.  While some may find the use of random sets intuitively appealing, others surely will be bewildered.  Therefore, we offer here an alternative, more direct, and arguably simpler construction that bypasses consideration of random sets.  The key to this simplification is the fact that our recommended random sets are {\em nested}, a characteristic that provides valuable additional structure compared to random sets which are not nested.  Roughly speaking, nested random sets are equivalent to {\em possibility measures} \citep[e.g.,][]{dubois.prade.book, zadeh1978}, and the latter can be characterized by ordinary functions---what we call {\em possibility contours}---rather than non-additive set functions.  In addition to their relative simplicity, possibility measures have a surprisingly close and inherent connection with the validity property we seek in the context of statistical inference.  This stems from a beautiful characterization of a possibility measure's credal set in terms of probabilities assigned to possibility contour's level sets. Consequently, possibility measures are fundamental to statistical inference, so drawing this connection provides some important insights beyond the simplified IM construction.  

The remainder of this paper is organized as follows.  After setting up the notation, the need for non-additivity, and the basic IM construction in Section~\ref{S:background}, we introduce the necessary concepts and terminology from possibility theory in Section~\ref{SS:possibility}.  Then Section~\ref{SS:imposs} develops the new possibility measure-based construction, various implications are discussed in Section~\ref{SS:remarks}, and its equivalence to the original IM formulation is established in Section~\ref{SS:equivalence}. Two illustrative examples are presented in Section~\ref{S:examples} with an emphasis on both the possibility measure interpretation and the key dimension reduction steps originally presented in \citet{imcond, immarg}.  Finally, Section~\ref{S:discuss} gives some concluding remarks and poses some open questions.

\section{Background}
\label{S:background} 

\subsection{Probabilistic inference and false confidence}
\label{SS:fc}

Statisticians, engineers, and others choose to quantify their uncertainties in terms of probabilities for various reasons.  And centered around virtually every one of those individual reasons a framework for probabilistic inference has been developed.  So what makes the IM framework unique?  Like many, we were originally troubled by the Bayesians' need of a prior distribution and we started down that now well-worn path, but we eventually realized the problem went much deeper.  Specifically, all existing frameworks for probabilistic inference fail to provide certain guarantees that we believe are crucial to the interpretation of those probabilities.  We start with a brief description of this observation.  

To set the scene, let $Y$ denote the observable data, with statistical model $Y \sim \prob_{Y|\theta}$, where $\prob_{Y|\theta}$ is a probability distribution for $Y$, supported on (a sigma-algebra containing subsets of) the space $\YY$, depending on a parameter $\theta \in \Theta$.  Now let $Q_y$ be a data-dependent probability measure defined on $\Theta$---a Bayes or empirical Bayes posterior \citep{berger1985, ghosh-etal-book}, a fiducial or generalized fiducial distribution \citep{fisher1935a, hannig.review, zabell1992}, a confidence distribution \citep{schweder.hjort.book, xie.singh.2012}, or something else---based on which inferences will be drawn.  That is, the truthfulness of assertions $A \subseteq \Theta$ concerning the unknown $\theta$ will be assessed based on the magnitude of $Q_y(A)$.  Therefore, in the spirit of replicability, a minimal requirement is that $Q_y$ assigning high probability to a false assertion should be a rare event.  More precisely, if $\theta \not\in A$, then $Q_Y(A)$, as a function of $Y \sim \prob_{Y|\theta}$, should not tend to be large.  Since $Q_y$ does not return ``real probabilities'' \citep[][p.~249]{fraser.copss}, if the aforementioned requirement is not met, and inferences are systematically misleading, then the $Q_y$-probabilities are not meaningful in any sense.  \citet{balch.martin.ferson.2017} showed that every data-dependent probability measure is afflicted with {\em false confidence}, i.e., there always exists false assertions that tend to be assigned high probability; see Section~\ref{SS:eeiv} below for an example, and \citet{martin.nonadditive} for more discussion.  Of course, not all assertions are afflicted with false confidence, and those that are afflicted might not be ``practical.''   But since there are examples where practically relevant assertions are afflicted with false confidence, including the satellite collision application in \citet{balch.martin.ferson.2017}, and no assurance that risks of ``systematically misleading conclusions'' are under control, this must be taken seriously.

\subsection{Basic inferential models}
\label{SS:basics}

The only way to manage false confidence is to abandon the use of probability for uncertainty quantification in the context of statistical inference.  But if not probability, then what?  Our proposal in \citet{imbasics, imbook} was to construct a non-additive/imprecise probability, in particular, a special type of belief function \citep[e.g.,][]{shafer1976, dempster1967, dempster1968b, kohlas.monney.hints} derived from the distribution of a (nested) random set \citep{nguyen.book, molchanov2005}.  The proposed construction was broken down into three steps: {\em associate} (A), {\em predict} (P), and {\em combine} (C).  

\begin{astep}
Associate the observable data $Y \in \YY$ and unknown parameter $\theta \in \Theta$ to an unobservable auxiliary variable $U \in \UU$, with known distribution $U \sim \prob_U$, via the relation 
\begin{equation}
\label{eq:baseline}
a(Y,\theta,U) = 0, 
\end{equation}
where the mapping $a$ is known.  This basically boils down to an algorithm for simulating data $Y \sim \prob_{Y|\theta}$ on a computer.  
\end{astep}

\begin{pstep}
Predict the unobserved value of $U$, associated with $(Y,\theta)$ in \eqref{eq:baseline}, with a random set $\S \subseteq \UU$, with distribution $\prob_\S$; more details on this below. 
\end{pstep}

\begin{cstep}
Combine the observed data $Y=y$, the association \eqref{eq:baseline}, and the random set $\S \sim \prob_\S$ into a new data-dependent random set $\Theta_y(\S) \subseteq \Theta$, given by 
\[ \Theta_y(\S) = \bigcup_{u \in \S} \{\vartheta: a(y,\vartheta,u) = 0\}. \]
The inferential output is determined by the distribution of $\Theta_y(\S)$, as a function of $\S \sim \prob_\S$.  
\end{cstep}

There are a number of relevant summaries of the distribution of $\Theta_y(\S)$.  In particular, the {\em belief function} or lower probability is 
\[ \bel_y(A) = \prob_\S\{\Theta_y(\S) \subseteq A\}, \quad A \subseteq \Theta, \]
and the {\em plausibility function} or upper probability is 
\[ \pl_y(A) = \prob_\S\{\Theta_y(\S) \cap A \neq \varnothing\}, \quad A \subseteq \Theta. \]
It is not difficult to see that $\lPi_y$ and $\uPi_y$ are dual:
\begin{equation}
\label{eq:dual}
\uPi_y(A) = 1-\lPi_y(A^c), \quad A \subseteq \Theta. 
\end{equation}
A couple technical remarks are in order.  First, note that the belief and plausibility functions are non-additive, e.g., $\bel_y(A) + \bel_y(A^c) \leq 1$, from which it follows that $\bel_y(A) \leq \pl_y(A)$, hence the interpretation as lower and upper probabilities.  More details on the imprecise probability structure of the inferential output will be given in Section~\ref{S:new}.  Second, if there is positive $\prob_\S$-probability that $\Theta_y(\S)$ is empty, then the above are replaced by conditional probabilities given ``$\Theta_y(\S) \neq \varnothing$.''  Non-emptiness can fail when there are non-trivial constraints on the parameter space \citep{leafliu2012}, but it is most common when the auxiliary variable dimension can be reduced; see Section~\ref{S:examples}.  

While the user's choice of random set $\S \sim \prob_\S$ is quite flexible, a good choice is critical to the interpretation and properties of the IM output.  The key is to make a connection between the distribution of the random set $\S$ and that of the unobservable auxiliary variable $U$.  Towards this, define the random set's {\em hitting probability} as 
\[ \gamma(u) = \prob_\S(\S \ni u), \quad u \in \UU. \]
Clearly, this depends on $\prob_\S$, and we make a link to $\prob_U$ by requiring that 
\begin{equation}
\label{eq:valid.prs}
\text{$\gamma(U) \stgeq \unif(0,1)$ as a function of $U \sim \prob_U$.} 
\end{equation}
The relation $\stgeq$ is read ``stochastically no smaller than'' and, in this case, means that the distribution function of $\gamma(U)$ is on or below that of $\unif(0,1)$.  This condition is mild and easy to arrange and we discuss this further below.  

With this, it can be shown that the IM output is {\em valid} in the sense of the following proposition.  In words, validity guarantees that assigning large belief to false assertions---or small plausibility to true assertions---is a rare event.  Consequently, unlike any probability-based framework (see Section~\ref{SS:fc}), IMs are not afflicted by false confidence.  

\begin{prop}
\label{prop:valid.old}
If the random set $\S \sim \prob_\S$ satisfies \eqref{eq:valid.prs}, then the IM output is valid in the sense that 
\begin{equation}
\label{eq:valid.bel}
\sup_{\theta \not\in A} \prob_{Y|\theta}\{\bel_Y(A) \geq 1-\alpha\} \leq \alpha, \quad \text{for all $A \subseteq \Theta$ and all $\alpha \in [0,1]$}. 
\end{equation}
Since this holds for all $A$, the duality \eqref{eq:dual} implies the equivalent property:
\begin{equation}
\label{eq:valid.pl}
\sup_{\theta \in A} \prob_{Y|\theta}\{\pl_Y(A) \leq \alpha\} \leq \alpha, \quad \text{for all $A \subseteq \Theta$ and all $\alpha \in [0,1]$}.
\end{equation}
\end{prop}

The validity theorem establishes a meaningful interpretation for the IM belief and plausibility output.  Beyond that, it provides desirable frequentist guarantees for decision procedures derived from the same IM output.  

\begin{prop}
\label{eq:freq}
Let $(\bel_Y, \pl_Y)$ be valid IM output as described above.   
\begin{enumerate}
\item[{\em (a)}] Consider testing $H_0: \theta \in A$, for any $A \subseteq \Theta$.  Then, for any $\alpha \in (0,1)$, the test that rejects $H_0$ if and only if $\pl_Y(A) \leq \alpha$ has Type~I error probability bounded by $\alpha$.  
\vspace{-2mm}
\item[{\em (b)}] For any $\alpha \in (0,1)$, the $100(1-\alpha)$\% plausibility region 
$\{ \vartheta \in \Theta: \pl_y(\{\vartheta\}) > \alpha \}$ 
has coverage probability lower-bounded by $1-\alpha$.
\end{enumerate}
\end{prop}

There is no shortage of, e.g., tests that can control Type~I errors for a given class of hypotheses. But if the class of hypotheses changes, then the test procedure changes too.  The IM-based tests, on the other hand, provide Type~I error control for all hypotheses.  In addition to tests and confidence regions, valid predictions can also be achieved \citep{impred}.  Finally, aside from producing valid solutions for a wide range of statistical problems, characterization results \citep[e.g.,][]{impval, imchar, imconformal} show that, roughly, for any valid test or confidence region, there exists a valid IM whose corresponding test or confidence region is at least as efficient.

\section{New IM construction}
\label{S:new}

\subsection{Possibility measures}
\label{SS:possibility}

Towards a more direct IM construction, one that apparently skips the random set specification, we provide the necessary background on a special imprecise probability model, namely, {\em possibility measures}.  Key references include \citet{dubois.prade.book}, \citet{cooman.1997a}, \citet{dubois2006}, and the chapter by \citet{destercke.dubois.2014} in the introductory volume \citet{imprecise.prob.book}.  Possibility measures also have close connections to Shafer's consonant belief functions and to fuzzy sets \citep[e.g.,][]{zadeh1978}.  

A possibility measure on a space $\UU$ is determined by a function $\pi: \UU \to [0,1]$, which we call the {\em possibility contour}, that satisfies $\sup_u \pi(u) = 1$.  This defines a {\em possibility measure} $\upi$, a set function defined on the power set of $\UU$, according to the rule
\begin{equation}
\label{eq:supremum}
\upi(K) = \sup_{u \in K} \pi(u), \quad K \subseteq \UU, 
\end{equation}
where supremum over the empty set is defined to be 0.  Clearly, $\upi$ satisfies $\upi(\varnothing)=0$ and $\upi(\UU)=1$ by definition, just like a probability measure, but it is not additive.  Indeed, possibility measures are supremum-preserving \citep[e.g.,][]{cooman.aeyels.1999}, i.e., 
\[ \upi\bigl( {\textstyle\bigcup_\lambda K_\lambda} \bigr) = \sup_\lambda \upi(K_\lambda), \quad \text{for any collection $K_\lambda \subseteq \UU$}. \]
Restricting to countable collections and using the fact that sums of non-negative numbers are bigger than suprema, it follows that $\upi$ is subadditive.  A possibility measure $\upi$ has a dual $\lpi$, called a {\em necessity measure}, given by $\lpi(K) = 1-\upi(K^c)$, and it defines a coherent lower prevision \citep{walley1991,lower.previsions.book}.  

An interesting class of possibility measure examples---in fact, the only class relevant to use here, see \eqref{eq:phi.P} below---are those determined by a pair $(\prob,h)$, where $\prob$ is probability distribution and $h$ is real-valued function, both on the same space $\UU$.  Then 
\begin{equation}
\label{eq:pi.h}
\pi(u) = \prob\{h(U) < h(u)\}, \quad u \in \UU, 
\end{equation}
is a possibility contour, and the corresponding $\upi$ is defined via optimization in \eqref{eq:supremum}. 

Of all the imprecise probability models, possibility measures are among the most restrictive---in particular, necessity and possibility measures are special cases of Shafer's belief and plausibility functions.  Here ``restrictive'' is a positive quality because it implies a level of simplicity that more general models do not have.  This makes it possible to answer questions such as how close a possibility measure is to a probability measure, etc.  Towards this, like any other imprecise probability model, a possibility measure $\upi$ determines a {\em credal set}, or set of compatible probability measures, 
\begin{equation}
\label{eq:credal}
\C(\upi) = \{\prob \in \text{prob}(\UU): \text{$\prob(K) \leq \upi(K)$ for all measurable $K$}\}, 
\end{equation}
where $\text{prob}(\UU)$ denotes the set of all probability measures on (a specified sigma-algebra of subsets of) $\UU$.  There are interesting characterizations of the property ``$\prob \in \C(\upi)$'' in terms of the $\prob$-probability assigned to the so-called $\alpha$-cuts of $\upi$.  In particular, if $\pi$ is the possibility contour, then the $\alpha$-cut of $\upi$ is defined as 
\[ C_{\upi}^\alpha = \{u: \pi(u) \geq \alpha\}, \quad \alpha \in [0,1]. \]
These are simply the upper level sets of $\pi$.  Then \citet{dubois.etal.2004} and, in a more general context, \citet{cuoso.etal.2001} established that 
\begin{equation}
\label{eq:implication}
\prob \in \C(\upi) \iff \prob(C_{\upi}^\alpha) \geq 1-\alpha \quad \text{for all $\alpha \in [0,1]$}. 
\end{equation}
Questions about ``how imprecise'' a given possibility measure $\upi$ is can then be addressed by looking at how diverse is the set of $\prob$ that satisfy the condition on the right-hand side of \eqref{eq:implication}.  That is, if very different $\prob$ are such that $\prob(C_{\upi}^\alpha) \geq 1-\alpha$ for all $\alpha$, then the credal set $\C(\upi)$ is large and $\upi$ is rather imprecise.  

Conversely, if there is a specific probability measure $\prob$ of interest, then \eqref{eq:implication} can be used to answer questions about ``how precise'' a possibility measure can be and remain compatible with the given $\prob$.  This latter question is particularly relevant to us.  The idea is to first define a measure of a possibility measure's precision or {\em specificity}, and then solve the corresponding maximization problem.  This is the {\em maximum specificity principle} \citep[e.g.,][]{dubois.prade.1986}.  Intuitively, contours $\pi$ being pointwise smaller gives the possibility measure more specificity, so the goal is to make this contour function small while maintaining the $\prob$-probability assignments to the corresponding $\alpha$-cuts.  If $\prob$ has a unimodal density function $f$, then the goal is to let $\pi$ match the shape of $f$ as closely as possible.  \citet{dubois.etal.2004} showed that the maximum specificity possibility distribution compatible with $\prob$ has contour
\begin{equation}
\label{eq:phi.P}
\pi_\prob(u) = \prob\{f(U) < f(u)\}. 
\end{equation}
Note how simple the solution to this complex optimization problem is; see, also, \eqref{eq:pi.h}.  Computation of $\pi_\prob$ is also relatively straightforward: it can be done exactly for nice distributions $\prob$, numerically (e.g., via quadrature) for complex $\prob$ in low dimensions, and via Monte Carlo in higher dimensions.  The most important property is a curious consequence of the probability assignments to $\alpha$-cuts, namely, 
\begin{equation}
\label{eq:valid.pre}
\text{if $U \sim \prob$, then $\pi_\prob(U) \stgeq \unif(0,1)$}. 
\end{equation}
For continuous data problems, if there are no sets of positive $\prob$-measure where $f$ is constant, then the above stochastic inequality is an equality.  We say this property is ``curious'' because it suggests an inherent connection between possibility measures and the validity property \eqref{eq:valid.prs} described above.

\subsection{IMs from possibility measures}
\label{SS:imposs}


For a given baseline association \eqref{eq:baseline}, let $\UU_y(\theta) = \{u: a(y,\theta,u)=0\}$ for a given $(y,\theta)$.  In continuous-data problems, this will be a singleton, which we denote as $u_{y,\theta}$.  
We do not require that the association equation can be solved for $\theta$, as a function of $(y,u)$, as would be needed to derive, say, a fiducial or structural distribution.  However, being able to reduce the dimension so that a solution for $\theta$ can be found is important, for efficiency purposes.  Details about dimension reduction are discussed in \citet{imcond, immarg}, and we summarize this in the context of two examples in Section~\ref{S:examples}.   

Our key assumption here is that $\bigcup_\theta \UU_y(\theta) = \UU$ for all $y$.  In words, this means that there are no constraints on the possible $u$ values induced by any observation $y$. 
This assumption might fail, for example, in problems that involve non-trivial constraints on the parameter space \citep{leafliu2012}.  In any case, if there are $u$ values that can be ruled out based on an observation $Y=y$, then that information would be known to the data analyst and should be used to sharpen his/her inference.  Here we are simply assuming that no such side information is available.  

Returning to the IM construction, start with the same baseline association as in \eqref{eq:baseline} and auxiliary variable distribution $\prob_U$, but now consider these new P- and C-steps.

\begin{pstep}
Model the post-data uncertainty about the unobserved value of $U$ by a possibility measure compatible with $\prob_U$.  Reasonable choices include those with contours like in \eqref{eq:pi.h}, in particular, the maximum specificity contour
\begin{equation}
\label{eq:pi.opt}
\pi(u) = \prob_U\{f(U) < f(u)\}, \quad u \in \UU, 
\end{equation}
with $f$ the density corresponding to $\prob_U$.  
\end{pstep}

\begin{cstep}
Combine the data $Y=y$, connection to $\theta$ and the auxiliary variable in association \eqref{eq:baseline}, with the above possibility measure to get a {\em posterior possibility contour}  
\[ \pi_y(\vartheta) = \sup_{u \in \UU_y(\vartheta)} \pi(u), \quad \vartheta \in \Theta. \]
\end{cstep}

From the posterior possibility contour $\pi_y$, we can extend to a posterior possibility measure defined on all subsets of $\Theta$ according to the optimization rule above:
\[ \upi_y(A) = \sup_{\vartheta \in A} \pi_y(\vartheta), \quad A \subseteq \Theta. \]
Define the corresponding posterior necessity measure as $\lpi_y(A) = 1-\upi_y(A^c)$.  This construction holds even for multivariate parameters, so the above formula can be used for marginal inference on any relevant feature $\phi = \phi(\theta)$ of the full parameter.  

There are some parallels between this and Fisher's fiducial inference; see Remark~\ref{re:fiducial} below.  But there are some obvious differences too.  One important difference is that a strong validity result like Proposition~\ref{prop:valid.old} above can easily be proved.  

\begin{prop}
\label{prop:valid}
The posterior possibility measure $\upi_Y$ is valid in the sense of \eqref{eq:valid.pl}, i.e.,  
\[ \sup_{\theta \in A} \prob_{Y|\theta}\{\upi_Y(A) \leq \alpha\} \leq \alpha, \quad \text{for all $A \subseteq \Theta$ and all $\alpha \in [0,1]$}. \]
Similarly, a claim analogous to \eqref{eq:valid.bel} holds for the posterior necessity measure $\lpi_Y$.   
\end{prop}

\begin{proof}
First, we have $\upi_Y(A) \geq \pi_Y(\theta)$ by monotonicity.  Second, $\pi_Y(\theta)$ equals $\pi(u_{Y,\theta})$, which, as a function of $Y \sim \prob_{Y|\theta}$, has the same distribution as $\pi(U)$.  Finally, the latter random variable is stochastically no smaller than $\unif(0,1)$ by \eqref{eq:valid.pre}.
\end{proof}

The possibility contour based on the maximum specificity criterion is persuasive, but it is not the only option.  In fact, there are some cases where this is not an option at all.  For example, if $\prob_U = \unif(0,1)$, which often appears in applications, then the density is constant---not unimodal---so the above version of the maximum specificity construction cannot be used.  Fortunately, \citet{dubois.etal.2004} show that the most precise possibility measure compatible with a symmetric probability distribution having bounded support is the so-called {\em triangular} possibility measure.  For symmetric distributions on $[0,1]$, the triangular possibility contour is given by 
\begin{equation}
\label{eq:triangular}
\pi(u) = 1 - |2u-1|, \quad u \in [0,1]. 
\end{equation}
It is easy to check that, with $\prob_U = \unif(0,1)$, this contour satisfies the important distributional property \eqref{eq:valid.pre}.  Other choices as in \eqref{eq:pi.h} are possible, 
but the benefit of using \eqref{eq:pi.opt} or \eqref{eq:triangular} is that these are principled choices, motivated by optimality.  

For a quick illustration, consider a single observation $Y$ from a Cauchy distribution with location parameter $\theta$.  Then the association is $Y - \theta - U = 0$, where $U \sim \prob_U$, with $\prob_U$ a standard Cauchy distribution, having density 
\[ f(u) \propto (1 + u^2)^{-1}, \quad u \in \RR. \]
An easy derivation shows that the ``optimal'' possibility measure has contour 
\[ \pi(u) = \prob_U\bigl\{ (1 + U^2)^{-1} < (1 + u^2)^{-1} \bigr\} = 2\{1 - F(|u|)\}, \]
where $F$ is the cumulative distribution function corresponding to $f$.  Since there is a unique solution $u_{y,\vartheta} = y-\vartheta$ for $u$, the posterior possibility contour $\pi_y$ is simple:
\[ \pi_y(\vartheta) = 2\{1 - F(|y-\vartheta|)\}. \]
The expression on the right-hand side is familiar to those who work with IMs; it is also related to the so-called {\em confidence curve} \citep{birnbaum1961} that often appears in the confidence distribution literature.  A plot of this function is displayed in Figure~\ref{fig:cauchy1}, for the case with $y=0$.  For example, the upper $\alpha$-level set, with $\alpha=0.05$, say, returns a 95\% confidence interval for $\theta$.  If one was interested in the hypothesis $A=(10,\infty)$, then it is clear from the contour's monotonicity that $\upi_y(A) = \pi_y(10)$ and $\lpi_y(A) = 0$.  

\begin{figure}[t]
\begin{center}
\scalebox{0.6}{\includegraphics{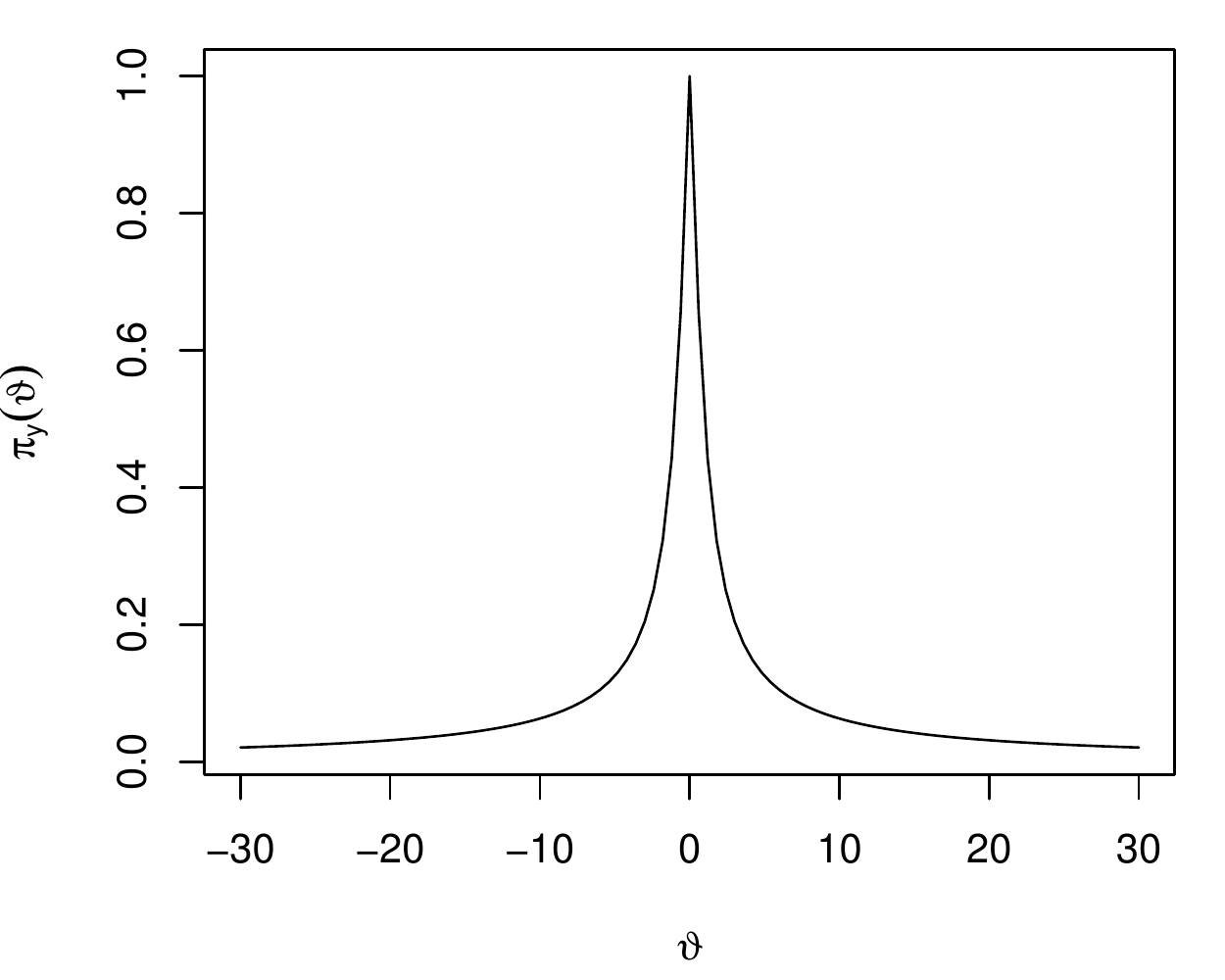}}
\end{center}
\caption{Possibility contour $\vartheta \mapsto \pi_y(\vartheta)$ for the Cauchy example with $y=0$.}
\label{fig:cauchy1}
\end{figure}

\subsection{Technical remarks}
\label{SS:remarks}

Fiducial inference, Dempster's extension, and IMs all share one thing in common: their use of the auxiliary variable with known distribution $\prob_U$.  The difference is that the former use $\prob_U$ directly to carry out inference while the latter adjusts $\prob_U$ either by introducing a random set (Section~\ref{SS:basics}) or a possibility measure (Section~\ref{SS:imposs}).  Here we make three remarks to compare and contrast fiducial and IMs when the latter are formulated via possibility measures.  Since the available IM optimality theory in \citet[][Sec.~4]{imbasics} relies on a connection to fiducial, the possibility measure formulation has some implications there too; see Remark~\ref{re:optimal}.  

\begin{remark}
\label{re:fiducial}
Various connections between this new IM construction and fiducial inference can be made.  Assume here that $(y,\theta,u)$ are all of the same dimension, so that we can solve the association equation for $u=u_{y,\theta}$ and $\theta=\theta_{y,u}$; see \citet{imcond} and Section~\ref{SS:curved} below for how this can be arranged.  Then the fiducial distribution for $\theta$, given $Y=y$, can be understood as the distribution for $\theta_{y,U}$, as a function of $U \sim \prob_U$, induced by probability calculus; this is the fiducial argument based on the {\em continue to regard} step in \citet{dempster1963} or the {\em switching principle} in \citet{hannig.review}.  What we described in Section~\ref{SS:imposs} is similar in spirit, but we replace ``$U \sim \prob_U$'' by a compatible possibility measure $\upi$ on $\UU$ and then propagate using possibility calculus.  An interesting question is if one could first obtain the fiducial distribution for $\theta$ using probability calculus---which, of course, is not valid---and then construct a compatible possibility measure from it to achieve validity.  Specifically, this latter approach would proceed by constructing the fiducial probability density function $\vartheta \mapsto f(u_{y,\vartheta}) J_y(\vartheta)$, where $J_y(\theta) = |\partial u_{y,\theta} / \partial\theta|$ is the Jacobian \citep[][Theorem~1]{hannig.review}, and then a compatible possibility measure with contour 
\[ \vartheta \mapsto \prob_U\{f(U) J_y(\theta_{y,U}) < f(u_{y,\vartheta}) J_y(\vartheta)\}. \]
The only way this can equal $\pi_y(\vartheta) = \prob_U\{f(U) < f(u_{y,\vartheta})\}$ is if the Jacobian term is constant, that is, if $u_{y,\vartheta}$ is linear in $\vartheta$ or, equivalently, if $\theta$ is a location parameter.  Therefore, in general, our proposal is {\em not} equivalent to first finding the fiducial distribution for $\theta$ and then constructing a compatible possibility measure.  
\end{remark}

\begin{remark}
\label{re:local}
Our strategy for finding $\upi$ is {\em global} in the sense that it focuses on a measure of compatibility that takes into consideration the probability/possibility assignments to all (measurable) subsets of $\UU$; see \eqref{eq:credal}.  Consider an alternative {\em local} strategy that focuses on a specified subclass $\mathscr{U}$ of subsets of $\UU$ and a weaker notion of compatibility between a probability $P$ and a possibility $\upi$, namely, 
\[ \text{$P(K) \leq \upi(K)$ for all $K \in \mathscr{U}$}. \]
We have in mind that $\mathscr{U}$ could be determined by a corresponding subclass of assertions $A \subseteq \Theta$ about $\theta$ deemed to be of special relevance, while all the other assertions dismissed as irrelevant.  In other words, this local strategy boils down to loosening the requirement that validity apply to {\em all assertions} about $\theta$, which, by the false confidence theorem, puts the reliability of inferences in jeopardy.  It turns out that the default-prior Bayes or fiducial solutions drop out as special cases of the above construction {\em only} under this risky local strategy.  To see this, consider a simple case where $(y,\theta,u)$ are all scalars, and let $\theta_{y,u}$ denote the $\theta$-solution to the association equation $a(y,\theta,u)=0$.  Suppose that the only assertions about $\theta$ of interest are half-lines in the class 
\[ \mathscr{A} = \{[\vartheta,\infty): \vartheta \in \RR\}. \]
If $u \mapsto \theta_{y,u}$ is monotone increasing, the above collection of half-lines on the parameter space corresponds to half-lines 
\[ \mathscr{U} = \{(-\infty,u]: u \in \RR\} \]
on the auxiliary variable space.  (If $u \mapsto \theta_{y,u}$ is monotone decreasing, then the orientation of the intervals in $\mathscr{U}$ reverses.) If the focus is exclusively on half-lines in the $u$-space, then the ``best'' compatible possibility measure is exactly the probability distribution, i.e., $\upi = \prob_U$.  With this choice, it is easy to see that the posterior possibility satisfies 
\[ \upi_y\bigl( [\vartheta, \infty) \bigr) = \prob_U(\theta_{y,U} \geq \vartheta). \]
The right-hand side is exactly Fisher's fiducial probability of the assertion ``$\theta \geq \vartheta$'' and, at least in transformation models, this would agree with the default-prior Bayesian posterior probability and Fraser's structural probability.  But the false confidence theorem applies to these probabilities, so the corresponding inference is not valid in the strong sense of Section~\ref{S:background}.  Thus, we recover the familiar result that fiducial inference is valid only in a limited sense, i.e., only for half-line assertions.  
\end{remark}

\begin{remark}
\label{re:optimal}
In the original IM framework, summarized in Section~\ref{SS:basics}, an important open question concerns {\em optimality}, i.e., what is the ``best'' random set $\S \sim \prob_\S$?  The specific optimality results presented in \citet{imbasics} are not fully satisfactory, but one of their preliminary results provides some valuable intuition.  Roughly, their Proposition~1 says that an IM cannot be more efficient than the primitive fiducial posterior from Remark~\ref{re:fiducial} above.  Of course, the latter is not valid, but the key take-away is that, for efficiency's sake, the IM's adjustment to achieve validity---via random sets or possibility measures---should be minimal.  Depending on the perspective, this boils down to choosing $\S$ as small as possible, in some sense, relative to the constraint \eqref{eq:valid.prs}, or choosing $\upi$ as close to $\prob_U$ as possible subject to the compatibility constraint.  An advantage of the new possibility measure-based approach is that the latter objective is apparently easier to formulate and solve mathematically, leading to $\upi$ defined by the contour \eqref{eq:pi.opt}.  Incidentally, in cases where $\prob_U$ has a unimodal density $f$, our conjecture of what the ``smallest possible $\S$ subject to \eqref{eq:valid.prs}'' would be is a random upper level set of $f$, i.e., 
\[ \S = \{u \in \UU: f(u) \geq f(\tilde U)\}, \quad \tilde U \sim \prob_U. \]
It is easy to check that this random set's hitting probability $u \mapsto \prob_\S(\S \ni u)$ is exactly the possibility contour in \eqref{eq:pi.opt}.  Therefore, the new possibility measure perspective provides some rigorous justification for our intuition-based random set recommendations.  Unfortunately, this possibility measure perspective does not provide a completely satisfactory answer to the question of optimality; see Section~\ref{S:discuss}.  
\end{remark}


\subsection{Equivalence of the two constructions?}
\label{SS:equivalence}

The distinguishing feature is that the original construction, as discussed in Section~\ref{S:background}, makes use of random sets while the new construction uses possibility measures.  A natural question is if the two are equivalent, at least in some sense. 

The fundamental admissibility result in \citet[][Theorem~3]{imbasics} implies that the only random sets to be considered in the IM framework are closed and {\em nested}; here, $\S$ is nested if, for any pair $S$ and $S'$ in the support of $\prob_\S$, either $S \subseteq S'$ or $S \supseteq S'$.  Examples of closed nested random sets include those with form like in the displayed equation of Remark~\ref{re:optimal}.  It is well known that hitting probabilities for nested random sets are possibility contour functions \citep[e.g.,][]{cooman.aeyels.2000}, so it is no surprise that, if we start from a nested random set $\S$, then we arrive at a possibility measure.  What about the other direction?  Suppose we start, as in Section~\ref{SS:imposs}, with a possibility measure $\upi$ on $\UU$, with contour $\pi$.  Is there a corresponding closed and nested random set whose hitting probability agrees with $\pi$?  Questions about the connections between different models are important, so there is guidance in the imprecise probability literature.  Indeed, \citet[][Theorem~4.4]{miranda.etal.2004} show that, in all practically relevant settings (e.g., $\UU$ a Euclidean space), every possibility measure corresponds to a closed nested random set.  Therefore, in our setting, when we directly introduce a possibility measure $\upi$ on $\UU$, with contour $\pi$, we have implicitly defined a nested random set $\S \sim \prob_\S$ such that $\pi(u) = \prob_\S(\S \ni u)$ for all $u$.  From this random set $\S$, the corresponding IM from Section~\ref{SS:basics} would have a plausibility function which, for any $A \subseteq \Theta$, satisfies  
\begin{align*}
\pl_y(A) & = \prob_\S\{\Theta_y(\S) \cap A \neq \varnothing\} \\
& = \prob_\S\{\text{$\S \ni u_{y,\vartheta}$ for some $u \in \UU_y(\vartheta)$ and some $\vartheta \in A$}\} \\
& = \sup_{\vartheta \in A} \sup_{u \in \UU_y(\vartheta)} \prob_\S(\S \ni u) \\
& = \sup_{\vartheta \in A} \sup_{u \in \UU_y(\vartheta)} \pi(u),
\end{align*}
where the third equality follows from $\S$ being nested, and the fourth by the above connection between $\pi$ and $\S$. We immediately recognize the right-hand side as $\upi_y(A)$, the possibility measure in Section~\ref{SS:imposs}, which establishes the equivalence.

Altogether, there is no loss (or gain) of generality in working with possibility measures instead of nested random sets.  However, we believe this alternative possibility-based construction is valuable for several reasons.  First, the use of random sets is a distinguishing feature of the IM framework, what leads to the important validity property, but this is admittedly complicated and unfamiliar to practitioners, so a formulation based on ordinary functions might be more easily accessible.  Second, since possibility measures are among the simplest of the imprecise probability models, drawing a connection between these and IMs can lead to some new insights or results.  Third, as discussed in Remark~\ref{re:optimal}, the existing theory for possibility measures gives rigorous confirmation of our intuition about a default multipurpose random set that is optimal in a certain sense.  

Pushing this equivalence inquiry further leads to some interesting open questions.  If $\S$ is nested, then it is easy to verify that the corresponding random set $\Theta_y(\S)$ on the parameter space is nested too, so the corresponding hitting probability is a possibility contour.  But, as before, the opposite direction is more interesting.  That is, suppose we directly define the posterior possibility measure $\upi_y$, without starting with $\upi$ on $\UU$ and propagating to $\Theta$.  Then the results of \citet{miranda.etal.2004} imply the existence of a closed, nested, and data-dependent random set $\mathcal{T}_y \sim \prob_{\mathcal{T}_y}$ on $\Theta$ that satisfies 
\[ \pi_y(\vartheta) = \prob_{\mathcal{T}_y}(\mathcal{T}_y \ni \vartheta), \quad \vartheta \in \Theta. \]
Then the question is if there exists a nested random set $\S$ on $\UU$ such that $\Theta_y(\S) = \mathcal{T}_y$?  Based on how the map from $\S$ to $\Theta_y(\S)$ is defined, the only reasonable candidate $\S$ would be of the form 
\[ \S = \S_y = \bigcup_{\vartheta \in \mathcal{T}_y} \{u_{y,\vartheta}\}, \]
which, in general, depends on data $y$.  One can imagine certain cases where the dependence on $y$ would naturally drop out, for example, in transformation models, but are there others?  And how to interpret a data-dependent random set?  Fraser has discussed data-dependent priors in Bayesian analysis \citep[e.g.,][]{fraser.reid.marras.yi.2010, fraser.etal.2016, fraser2011} and there could be a connection.  There is also a question about if and, if so, then under what conditions, would a data-dependent random set lead to valid inference?  One situation where validity can be achieved with data-dependent random sets is in \citet{leafliu2012}; see, also, \citet{imexpert}.  There is also a yet-to-be-fully-understood connection between the parameter-dependent random sets in \citet{imcond}, \citet{imchar}, etc., and their data-dependent counterparts.

\section{Examples}
\label{S:examples}

As indicated above, steps to reduce the dimension of the auxiliary variable space are crucial to making the IM solution (both valid and) efficient.  These dimension-reduction techniques are described in detail in \citet{imcond, immarg}.  Rather than give a general survey, we have opted to describe these methods in the context of two examples.  These will also serve as  illustrations of the new possibility measure-based approach.  

\subsection{Curved normal}
\label{SS:curved}

Let $Y$ consist of $n$ observations, independent and identically distributed according to $\nm(\theta, \theta^2)$, with $\theta \in \Theta = \RR$.  Since $\theta$ controls both the mean and variance, this normal model is said to be {\em curved}.  This curvature makes the problem non-regular and generally more challenging than the version with separate mean and variance parameters.  

We can immediately reduce the problem down to one that involves minimal sufficient statistics.  With a slight abuse of notation, let $Y_1$ and $Y_2$ denote the mean and standard deviation, respectively, of the original $n$-sample, and write the association \eqref{eq:baseline} as 
\[ Y_1 = \theta + |\theta| U_1 \quad \text{and} \quad Y_2 = |\theta| U_2, \]
with $\prob_U$ the joint distribution of $U_1 \sim \nm(0,n^{-1})$ and $U_2 \sim \{(n-1)^{-1} \chisq(n-1)\}^{1/2}$, independent.  Since $\theta$ is a scalar while $U=(U_1,U_2)$ is two-dimensional, there is an opportunity to reduce dimension further.  

For the moment, assume that $\sign(\theta) \in \{-1,+1\}$ is {\em known}; this is not an unrealistic assumption, but we will discuss below how it can easily be removed.  To reduce the dimension, our strategy is to first find a feature of $U$ whose value is actually {\em observed}, and then condition on the value of that observed feature.  Aside from problems with a group transformation structure, where certain invariant statistics can be identified, there are very few general methods for finding observed features of $U$.  One suggested in \citet{imcond} is based on solving a partial differential equation.  That is, we seek a function $\eta$ such that $\eta(u_{y,\theta})$ is not sensitive to changes in $\theta$, in other words, such that $\partial\eta(u_{y,\theta})/\partial\theta = 0$.  By the chain rule, this amounts to finding $\eta$ that satisfies 
\[ \frac{\partial \eta(u)}{\partial u} \Bigr|_{u=u_{y,\theta}} \, \frac{\partial u_{y,\theta}}{\partial\theta} = 0. \]
In the present case, since $u=u_{y,\theta} = |\theta|^{-1} ( y_1-\theta, \, y_2 )^\top$, and 
\[ \frac{\partial u}{\partial \theta} = -\frac{\sign(\theta)}{|\theta|} \begin{pmatrix} u_1 + \sign(\theta) \\ u_2 \end{pmatrix}, \]
where $\sign(\theta)$ is assumed known, it is easy to check that 
\[ \eta(u) = u_2^{-1} \{\sign(\theta) + u_1\} \]
solves the differential equation.  Therefore, with known $\sign(\theta)$, the quantity $\eta(U)$ is observed, and its value in the sample is 
\[ \eta(u_{y,\theta}) = (y_2/|\theta|)^{-1} \{\sign(\theta) + (y_1-\theta)/|\theta|\} = y_1/y_2. \]
If we let $H(y)=y_1/y_2$, then we have a reduced association 
\[ Y_1 = \theta + Y_2 \tau(U) \quad \text{and} \quad \eta(U) = H(Y), \]
where $\tau(u) = u_1/u_2$ is a scalar.  Then the proposal is to build an IM based on the conditional distribution of $V=\tau(U)$, given $\eta(U) = H(y)$.  If we write $h$ for the observed value of $H(Y)$, then that conditional distribution for $V$ has density $f_h$ that satisfies 
\[ \log f_h(v) = \log|\tfrac{1}{h-v}|-\tfrac{n}{2}(\tfrac{v}{h-v})^2+(n-2)\log\tfrac{\sign(\theta)}{h-v}-\tfrac{n-1}{2}(\tfrac{1}{h-v})^2 + \text{constant}, \]
for all $v$ such that $\sign(\theta) / (h-v) > 0$.  From here, construct a (conditional) IM for $\theta$ by defining the $f_h$-dependent possibility contour 
\[ \pi_h(v) = \prob_{\tau(U)|\eta(U)=h}\{f_h(\tau(U)) < f_h(v)\}, \]
and then the corresponding posterior version:
\[ \pi_{y|h}(\vartheta) = \pi_h(v_{y,\vartheta}), \]
where $v_{y,\vartheta} = (y_1 - \theta)/y_2$.  It follows from Theorem~1 in \citet{imcond} that the corresponding (conditional) IM is valid.  

For illustration, suppose data of size $n=10$ are sampled from the curved normal distribution with $\theta=2$.  The sample mean and standard deviation are $y_1=1.86$ and $y_2=2.12$, respectively, so that $h=H(y)=0.88$.  A plot of the posterior possibility contour is shown in Figure~\ref{fig:curved}.  The horizontal cut---and corresponding vertical fences---at $\alpha=0.05$ determines a 95\% confidence interval for $\theta$ based on these data.  To assess the IM's validity, we do a small simulation study, where we repeat the above experiment 1000 times.  For each data set, we produce 95\% confidence intervals based on the above IM and also a (higher-order accurate) generalized fiducial interval.  For the IM intervals, the estimated coverage probability and mean length are 0.956 and 1.82, respectively; for the fiducial intervals, the estimated coverage probability and mean length are 0.932 and 1.45, respectively.  On the one hand, the fiducial intervals are a little shorter, but they tend to under-cover.  On the other hand, the IM intervals are guaranteed to achieve the nominal coverage thanks to validity.  

\begin{figure}[t]
\begin{center}
\scalebox{0.6}{\includegraphics{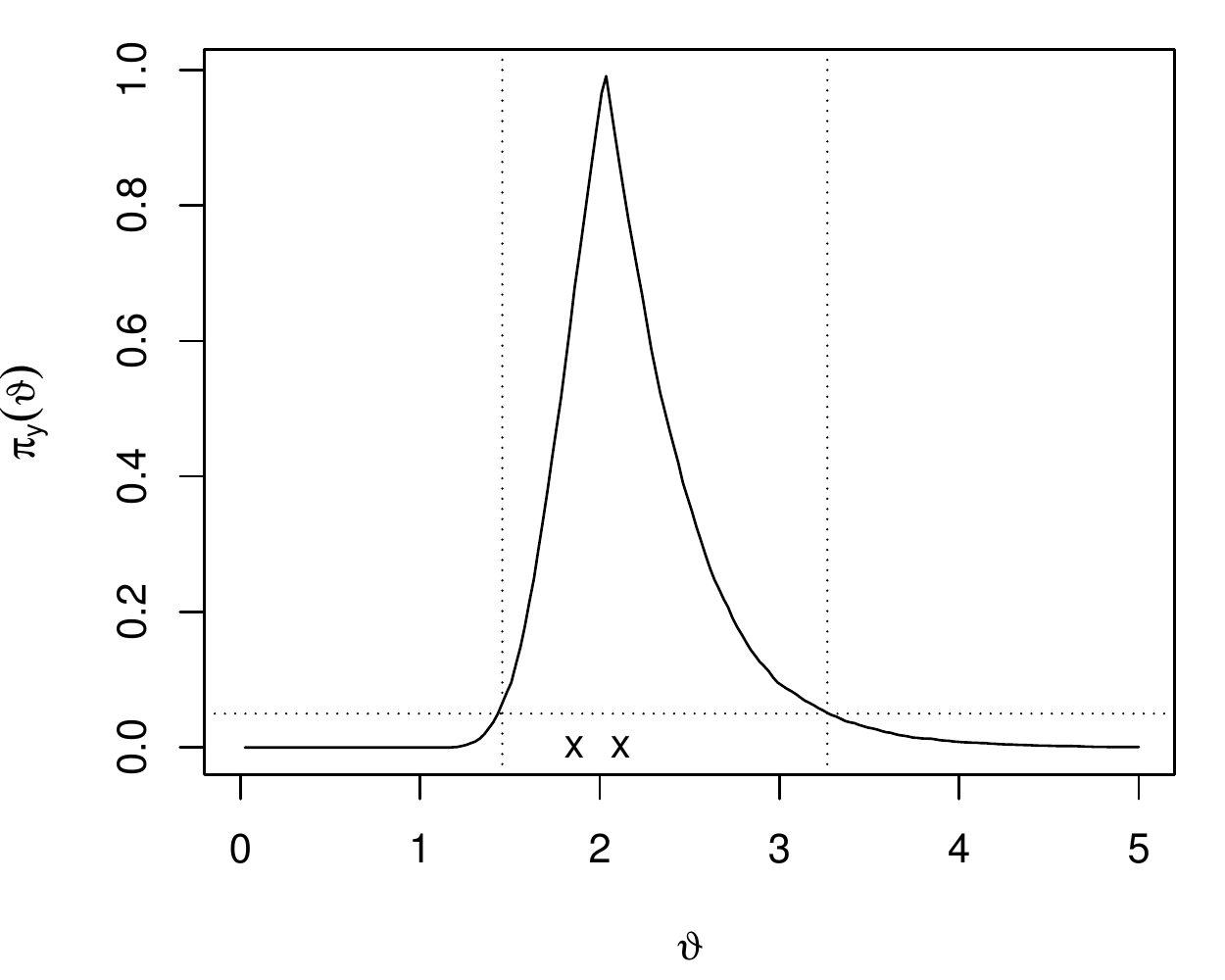}}
\end{center}
\caption{Possibility contour $\vartheta \mapsto \pi_{y|h}(\vartheta)$ for the curved normal example with $n=10$, $\theta=2$, and sample mean and standard deviation marked with {\sf X}.}
\label{fig:curved}
\end{figure}

As promised, it is possible to remove the assumption that $\sign(\theta)$ is known.  This requires a so-called {\em local} conditional IM in \citet{imcond}, allowing the differential equation's solution to depend on a local parameter, and then the local parameter-dependent IMs get suitably glued together.  Here the dependence on the local parameter is very weak---only depends on the sign---so this is relatively easy to manage.  But since assuming $\text{sign}(\theta)$ to be known is very mild, we opt for the simpler solution here.

\subsection{Exponential errors-in-variables}
\label{SS:eeiv}

Suppose we have a pair of independent observations $Y=(Y_1,Y_2)$, where 
\[ Y_i = \theta_i + U_i, \quad i=1,2. \]
The distribution $\prob_U$ of $U=(U_1,U_2)$ is known and here we consider the case where $U_i \sim \expo(\lambda_i)$, independent, but where the rate parameters, $\lambda_1$ and $\lambda_2$, need not be equal.  As an example, if there are $n_i$ independent and identically distributed exponentials shifted by $\theta_i$, then $Y_i$ represents the minimum of those and $\lambda_i = n_i$; this is what would emerge from the conditioning argument described in Section~\ref{SS:curved} above.  For joint inference on $\theta=(\theta_1,\theta_2)$, there is a density $f$ for $U$ and it is straightforward to follow the maximum specificity principle and construct a joint posterior possibility contour for $\theta$ as described above.  Similarly, for a location parameter problem like this one, a flat prior is the appropriate default choice, and one can immediately derive a Bayesian posterior distribution; the fiducial and structural distributions are the same.  

Things are less straightforward, however, if we add some additional structure.  Imagine that the support lower bound $\theta_1$ for $Y_1$ is believed to be a linear function of $\theta_2$, and the target is to estimate that function, but only a noisy measurement $Y_2$ of $\theta_2$ is available.  This boils down to a simple version of the {\em measurement error} problem, wherein the true ``covariate'' $\theta_2$ can only be measured with error.  From this point of view, it is natural to rewrite the problem as 
\[ Y_1 = \phi\xi + U_1 \quad \text{and} \quad Y_2 = \xi + U_2, \]
where $\phi > 0$ is the interest parameter and $\xi > 0$ is a nuisance parameter.  Then the goal is marginal inference on $\phi$ in the presence of unknown $\xi$ \citep[cf.][]{creasy1954, fieller1954}.  

Plugging the second equation in the above display into the first, we arrive at 
\[ Y_1 - \phi Y_2 = U_1 - \phi U_2 \quad \text{and} \quad Y_2 = \xi + U_2. \]
Notice that the first equation is free of $\xi$.  For such cases, the general theory in \citet{immarg} says that exact marginal inference can be achieved by simply ignoring the second $\xi$-dependent equation.  Not only does this eliminate the nuisance parameter, but it creates an opportunity for dimension reduction.  Since the distribution of $U_1 - \phi U_2$ is known, namely, {\em asymmetric Laplace}, we can employ a probability integral transformation.  That is, $U_1-\phi U_2$ has the same distribution as $G_\phi^{-1}(V)$, where $V \sim \unif(0,1)$ and $G_\phi = F_{\lambda_1,\phi \lambda_2}$ is the asymmetric Laplace distribution function parametrized as 
\[ F_{r_1,r_2}(x) = \begin{cases} 1 - \frac{r_2}{r_1 + r_2} e^{-r_1 x} & \text{if $x \geq 0$} \\ \frac{r_1}{r_1 + r_2} e^{r_2 x}, & \text{if $x < 0$}. \end{cases} \] 
If we use the triangular possibility contour \eqref{eq:triangular} for $V$, then the corresponding marginal posterior possibility contour for $\psi$ is 
\[ \pi_y(\varphi) = 1-\bigl| 2 G_\varphi(y_1 - \varphi y_2) - 1 \bigr|, \quad \varphi > 0. \]
It follows from Theorem~2 in \citet{immarg} that the corresponding IM is valid, which implies, for example, that confidence regions derived from this posterior possibility contour achieve the nominal coverage probability.  

For illustration, consider the case where $\lambda_1=\lambda_2=5$, $\theta_1=1$, and $\theta_2=0.1$; this implies $\xi=0.1$ and $\phi=10$.  Figure~\ref{fig:eeiv}(a) plots the posterior possibility contour.  Note that the peak of this plot is close to the true $\phi=10$, but that the curve does not vanish in the right tail.  This is because the data cannot rule out arbitrarily large values of $\phi$, hence a 95\% IM confidence interval would be unbounded.  This unboundedness is actually necessary to achieve the nominal coverage in this problem; see \citet{gleser.hwang.1987}.  To compare the IM with a Bayes solution, we can find the marginal posterior for $\phi = \theta_1/\theta_2$ based on a default flat prior for $(\theta_1,\theta_2)$, which we denote by $Q_y$.  Consider an assertion $A = \{\phi \leq 9\}$, which happens to be false in this example.  We compare the distribution of belief assignments to $A$ based on the Bayes and IM posteriors.  In particular, we compare the two distribution functions 
\[ \alpha \mapsto \prob_{Y|\theta}\{Q_Y(A) \leq \alpha\} \quad \text{and} \quad \alpha \mapsto \prob_{Y|\theta}\{\lpi_Y(A) \leq \alpha\}. \]
On the belief/lower probability/necessity scale, validity corresponds to this distribution function being on or above the diagonal line.  As expected, we find the IM curve above the diagonal line, while the Bayesian curve is well below.  Being so far below the diagonal line in this case is rather problematic because it means the Bayesian posterior probabilities assigned to the false assertion are {\em almost always large}.  In light of this false confidence, it is clear that the Bayesian inference is not reliable.   

\begin{figure}[t]
\begin{center}
\subfigure[Posterior possibility contour]{\scalebox{0.6}{\includegraphics{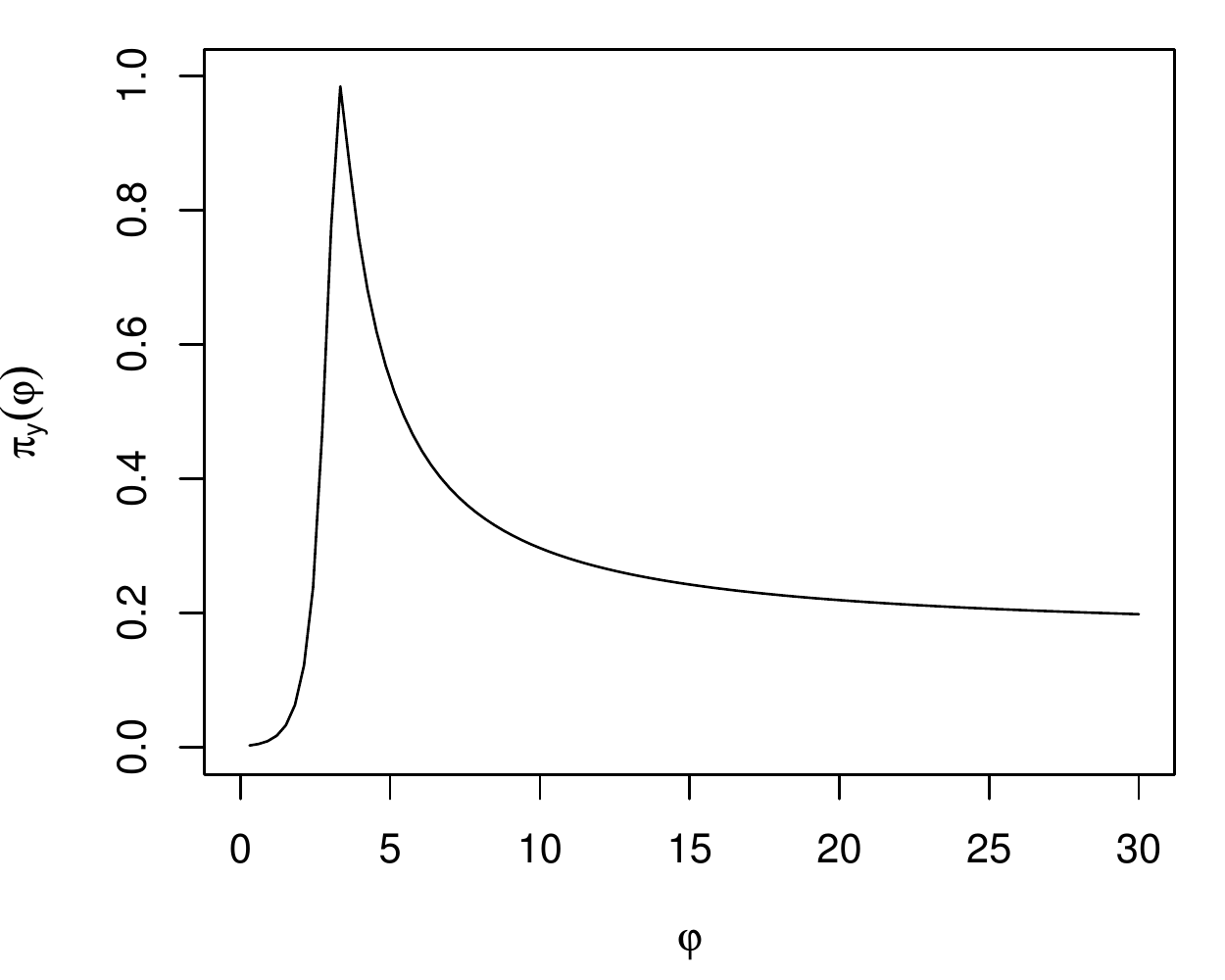}}}
\subfigure[Distribution function of $Q_Y(A)$ and $\lpi_Y(A)$]{\scalebox{0.6}{\includegraphics{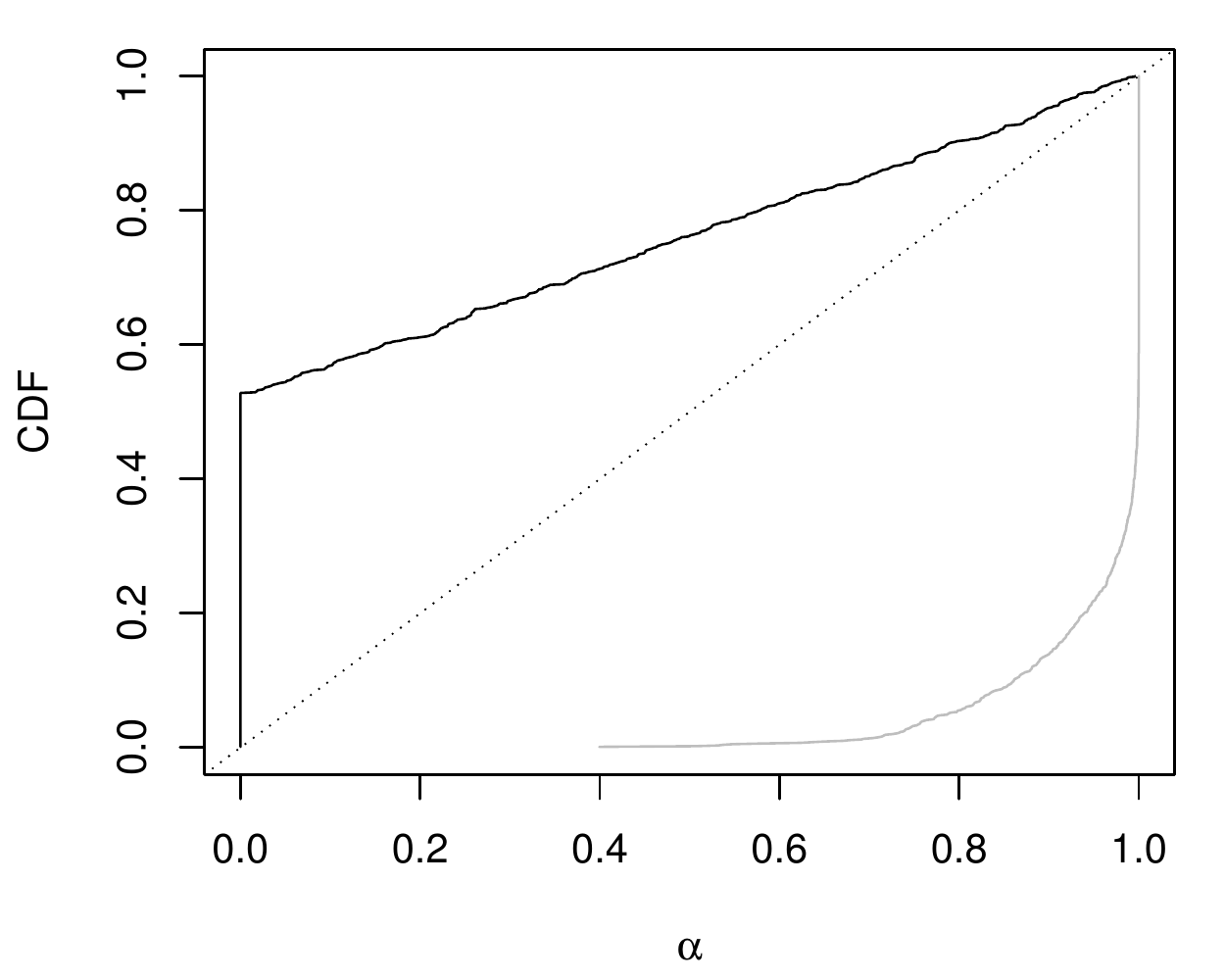}}}
\end{center}
\caption{Left: plot of the posterior possibility contour $\varphi \mapsto \pi_y(\varphi)$ in the exponential error-in-variables example, with true value $\phi=10$ and $(y_1,y_2) = (1.40,0.50)$.  Right: plots of the distribution function $\alpha \mapsto \prob_{Y|\theta}\{\star \leq \alpha\}$, for $\star$ equal to $Q_Y(A)$ (gray) and $\lpi_Y(A)$ (black), where $A = \{\phi \leq 9\}$ is a false assertion.}
\label{fig:eeiv}
\end{figure}


\section{Conclusion}
\label{S:discuss}

Here we have presented a new perspective on and construction of IMs based on possibility measures.  The chief advantage to this formulation is its relative simplicity, i.e., a user need not directly consider specification of and computation with a random set.  And this benefit comes without sacrificing on the essential validity property.  In fact, certain characterizations suggest that validity is somehow inherent in possibility measures, giving them a fundamental role in statistical inference.  Beyond validity, the existing possibility theory provides some guidance in terms of efficiency and optimality, as discussed in Remark~\ref{re:optimal}, and the dimension-reduction strategies for improved efficiency fit in seamlessly. There is certainly more work to be done, but we hope that the reader can see the IM framework's potential as a solution to Efron's ``most important unresolved problem.'' 

We conclude by mentioning a few ideas and open problems. 
\begin{itemize}
\item As mentioned in Remark~\ref{re:optimal}, the possibility measure connection does not completely resolve the optimality question.  A reason is that the maximum specificity based possibility measure apparently depends on the particular choice of $(a,U)$, association and auxiliary variable.  Intuitively, an ``optimal'' IM should depend only on the assumed statistical model, $Y \sim \prob_{Y|\theta}$.  Towards a better understanding of efficiency and optimality, it would be interesting to investigate the IM's sensitivity to different choices of $(a,U)$ when using the maximum specificity-based possibility measure.  
\vspace{-2mm}
\item When $\theta=(\phi, \xi)$ and $\xi$ is a nuisance parameter,in cases where $\xi$ cannot be completely eliminated, \citet{immarg} work with a $\xi$-dependent auxiliary variable and then suitably fatten the corresponding random set to achieve valid marginal inference on $\phi$.  In the context of the present paper, an alternative strategy for this would be to find the most precise $\upi$ that is compatible with {\em all} of the $\xi$-dependent auxiliary variable distributions.  Is the latter compatibility task easier or more efficient than the random set fattening?  
\vspace{-2mm}
\item In statistics, we often have data {\em and} some vaguely-specified structural assumptions, such as sparsity or smoothness.  These separate pieces of information are typically combined by introducing either a penalty or a Bayesian prior distribution.  But an important topic in imprecise probability is combining relevant information from different sources into an overall assessment of uncertainty.  If validity dictates that we operate in an imprecise probability domain, then we should investigate the properties of these alternative combination rules in our modern, high-dimensional statistical applications involving structural assumptions.  
\end{itemize}




\ifthenelse{1=1}{
\bibliographystyle{apalike}
\bibliography{/Users/rgmarti3/Dropbox/Research/mybib}
}{

}

\end{document}